\documentclass{amsart}
\usepackage{amsfonts}
\usepackage{graphicx}
\usepackage{graphics,color}
\usepackage{amsmath}
\usepackage{amssymb}
\usepackage{morefloats}

\newtheorem{theorem}{Theorem}
\theoremstyle{plain}

\newtheorem{corollary}{Corollary}

\newtheorem{definition}{Definition}

\numberwithin{equation}{section}
\input{tcilatex}

\begin{document}
\title[Energy Dissipation of Eddy Viscosity Models]{Diagnostics for eddy viscosity models of turbulence including data-driven/neural network based parameterizations}
\author{William Layton and Michael Schneier}
\address{Dept. of Mathematics, Univ. of Pittsburgh, Pittsburgh, PA 15260, USA%
}
\email{wjl@pitt.edu, mhs64@pitt.edu}
\urladdr{www.math.pitt.edu/\symbol{126}wjl }
\thanks{The research herein was partially supported by NSF grant DMS
1817542. }
\thanks{}
\dedicatory{ }

\begin{abstract}
Classical eddy viscosity models add a viscosity term with turbulent
viscosity coefficient whose specification varies from model to model.
Turbulent viscosity coefficient approximations of unknown accuracy are
typically constructed by solving associated systems of nonlinear evolution
equations or by data driven approaches such as deep neural networks. Often
eddy viscosity models over-diffuse, so additional fixes are added. This
process increases model complexity and decreases model comprehensibility,
leading to  the following two questions: \textit{Is an eddy viscosity model
needed?} \textit{Does the eddy viscosity model fail?}  This report derives \emph{a
posteriori} computable conditions that answer these two questions.
\end{abstract}

\maketitle

\section{Introduction}

In computational fluid dynamics, turbulence \cite{S01}, incomplete data,
quantification of uncertainty \cite{COPPM11}, a finite predictability
horizon \cite{TK93}, flow sensitivity \cite{MX06} and other issues lead to
the problem of computing \textit{averages} (denoted $u(x,t)$) of under
resolved (higher Reynolds number) solutions of the Navier-Stokes equations.
The most common approach \cite{W98}, among many, is to solve numerically
an eddy viscosity\footnote{%
The model arises after averaging (e.g., ensemble averages, time averages,
local space averaging) the NSE in which a non closed term arises. After
adjusting the pressure, eddy viscosity models replace that term by $-\nabla
\cdot \left( \nu _{turb}\nabla ^{s}u\right) $.} model 
\begin{equation}
u_{t}+u\cdot \nabla u-\nabla \cdot \left( \lbrack 2\nu +\nu _{turb}(\cdot
)]\nabla ^{s}u\right) +\nabla p=f(x)\text{ and }\nabla \cdot u=0
\label{eq:EVmodel}
\end{equation}%
subject to initial and boundary conditions. Here, $\nu $ is the kinematic
viscosity, $f(x)$ is the body force, $\nabla ^{s}u$ is the symmetric part of 
$\nabla u$, and $\nu _{turb}(\cdot )$\ is the eddy or turbulent viscosity. We let $U$ and $L$ denote a characteristic velocity and length scales respectively (defined precisely in
Section 2). The usual Reynolds number is then $\mathcal{R}e=LU/\nu $. This
holds in a $3d$, bounded, regular flow domain $\Omega $ subject to no-slip
boundary conditions ($u=0$ on $\partial \Omega $) and initial condition $%
u(x,0)=u_{0}(x)$. We assume $f(x)$ is smooth, $\nabla \cdot f=0$ in $\Omega $%
, and $f(x)=0$ on $\partial \Omega $.

Herein, $u(x,t)$ denotes the (sought) average velocity and, as usual, the
fluctuation about it is\textbf{\ }$u^{\prime }$. Thus, the induced turbulent
kinetic energy (TKE) density is $k^{\prime }(x,t):=\frac{1}{2}$ $|u^{\prime
}|^{2}(x,t)$. The Kolmogorov-Prandtl relation for $\mathbf{\nu }_{turb}$ is 
\begin{equation}
\mathbf{\nu }_{turb}\mathbf{(}l,k^{\prime }\mathbf{)}\text{ }=\sqrt{2}\mu l%
\sqrt{k^{\prime }},  \label{eq:KPrelation}
\end{equation}%
where $l(x,t)$\ has units of length (a mixing length or turbulent length
scale) and $\mu $\ is a calibration parameter. Determining $\nu
_{turb}(\cdot )$\ then reduces to modelling the unknowns $l,k^{\prime }$\ in
terms of computable flow variables and then calibrating $\mu $.

In all cases, two central questions, addressed herein via \textit{a
posteriori} computable conditions, arise: \textit{Is an eddy viscosity
necessary?} and \textit{Does the model fail?}

\textbf{Question 1. Is an eddy viscosity model necessary?}\textit{\ }%
Phenomenology and many numerical tests suggests that an under-resolved
simulation will be under-diffused and energy will accumulate in the smallest
resolved scale (non-physical $O(\triangle x)$\ oscillations).\textit{\ }The
classical interpretation has been that eddy viscosity is necessary if the
mesh does not resolve energetically significant eddies ($\triangle x\simeq 
\mathcal{R}e^{-3/4}L$, the Kolmogorov micro-scale). Answering question 1,
Theorem 1, Section 3 shows, surprisingly, that \textit{if the mesh resolves
the Taylor microscale (if }$\ \triangle x\simeq \sqrt{15}Re^{-1/2}L$\textit{%
) then the flow in the aggregate is not under diffused}.

\textbf{Question 2. Does the model fail? } Eddy viscosity models most commonly\footnote{%
Other failure modes, not considered herein, do occur intermittently when
reproducing observed flow phenomena requires brief intervals of negative
eddy viscosity values, resulting in numerical instabilities, Starr \cite{S68}%
. Simulations can also fail by having a correct aggregate model dissipation
but an incorrect distribution.}  fail by over damping
the solution, either producing a lower $\mathcal{R}e$\ flow or even driving
the solution to a nonphysical steady state. One can compute the aggregate
model dissipation, $\int \nu _{turb}|\nabla ^{s}u|^{2}dx$, and signal
failure if too large. (Like a diagnosis that a patient ``looks sick," this
offers little insight into the cause or its correction.) Theorem 2, Section
4 separates out the effect of the chosen turbulent viscosity
parameterization from the symmetric gradient, proving%
\begin{equation*}
\text{time-average model energy dissipation}\leq \left( \frac{1}{2}+\mathcal{%
R}e^{-1}+\frac{avg(\mathbf{\nu }_{turb})}{LU}\right) \frac{U^{3}}{L},
\end{equation*}%
where $avg(\cdot )$ denotes a space-time average defined precisely in
Section 2. The term $avg(\nu _{turb})/LU$ is \textit{a computable statistic
which, if }$O(1)$\textit{, implies the eddy viscosity model does not over
diffuse the flow}. From (\ref{eq:KPrelation}), $\mathbf{\nu }_{turb}$ has
two contributors: the parameterization of $l$ and $k^{\prime }$. Further,
Theorem \ref{thm:turb_intens1}, Section \ref{sec:EV_failure} shows $avg(%
\mathbf{\nu }_{turb})/LU=O(1)$ if $avg(l)/L=O(1)$ and the model's predicted
turbulent intensity (derived from the $k^{\prime }$ parameterization) $I_{%
\text{model}}=O(1)$ (see Section 4 for definitions). This follows from
estimate \eqref{eqn:sec4-1} in Theorem \ref{thm:turb_intens2}, Section \ref%
{sec:EV_failure}:%
\begin{equation*}
\text{ }\frac{avg(\mathbf{\nu }_{turb})}{LU}\leq \mu \frac{avg(l)}{L}\sqrt{%
I_{\text{model}}}\text{,}
\end{equation*}%
indicating the evolution of \textit{the model length scale and the model's
predicted turbulent intensity are determining statistics to monitor}. The
importance of this result is that the three computable quantities 
\begin{equation*}
\frac{avg(\mathbf{\nu }_{turb})}{LU},\frac{avg(l)}{L},I_{\text{model}},
\end{equation*}%
can all be monitored in a calculation. As long as they are $O(1)$, the
aggregate eddy viscosity is not over dissipating the (aggregate) flow. If
too large, their spacial distribution can be checked and the resulting
information used to isolate the cause and improve its parameterization.

{\textbf{Neural network (NN) based parameterizations} have seen an explosion of interest in determining these quantities, e.g., \cite{LKT16}, \cite{SALL19}. While NN based approximations have been successful, they lack theoretical guarantees of stability and convergence. These statistics
can be used to indicate the need to retrain the network or incorporated as a constraint into the training procedure.}

We therefore consider the eddy viscosity model (\ref{eq:EVmodel}). Let $%
||\cdot ||$ denote the usual $L^{2}$ norm. Taking the dot product with the
solution and integrating in space and time shows that a classical solution
satisfies the energy equality (e.g. \cite{DG95})%
\begin{gather}
\frac{1}{2}||u(T)||^{2}+\int_{0}^{T}\int_{\Omega }[2\nu +\nu
_{turb}(x,t)]|\nabla ^{s}u(x,t)|^{2}dxdt=  \label{eq:EnergyEquality} \\
=\frac{1}{2}||u_{0}||^{2}+\int_{0}^{T}(f,u(t))\,dt.  \notag
\end{gather}%
The model's space-averaged energy dissipation rate is thus $\varepsilon
=\varepsilon _{0}+\varepsilon _{turb}$ where%
\begin{equation*}
\varepsilon _{0}=\frac{1}{|\Omega |}\int_{\Omega }2\nu |\nabla
^{s}u(x,t)|^{2}dx\text{ and }\varepsilon _{turb}=\frac{1}{|\Omega |}%
\int_{\Omega }\nu _{turb}(x,t)|\nabla ^{s}u(x,t)|^{2}dx.
\end{equation*}%
We assume that solutions exist for the model and satisfy a standard energy
inequality. There has been slow but steady progress on an existence theory
for eddy viscosity models, summarized in Chac\'on-Rebollo and Lewandowski \cite%
{RL14}, but many open questions remain since the number of models seems to
be increasing faster than their analytic foundations develop.

\textbf{Assumption:} \textit{We assume that weak solutions of (\ref%
{eq:EVmodel}) exist\footnote{%
Even in the absence of a complete existence theory, the analysis of energy
dissipation rates can be performed for variational discretizations in space
(such as finite element methods or spectral methods). The same sequence of
steps shows that the discrete solutions satisfy the same energy dissipation
rate bounds uniformly in any space discretization parameter (such as mesh
width or frequency cutoff). Since the primary utility of turbulence models
is to account for breaking the communication between the inertial range and
dissipation range in numerical simulations\ after space discretizations,
this analysis is highly relevant for the uses of the models. It however adds
significant notational complexity without requiring any new mathematical
ideas or even steps, we shall assume the above about the continuum model for
purposes of greater clarity.} for any divergence free }$u_{0},f\in L^{2}$%
\textit{\ and satisfy the energy inequality}%
\begin{gather}
\frac{1}{2}\frac{1}{|\Omega |}||u(T)||^{2}+  \label{eq:EnergyInequality} \\
+\int_{0}^{T}\frac{1}{|\Omega |}\int_{\Omega }2\nu |\nabla
^{s}u(x,t)|^{2}+\nu _{turb}(x,t)|\nabla ^{s}u(x,t)|^{2}dxdt  \notag \\
\leq \frac{1}{2}\frac{1}{|\Omega |}||u_{0}||^{2}+\int_{0}^{T}\frac{1}{%
|\Omega |}(f,u(t))\,dt.  \notag
\end{gather}

\subsection{Related work}

The energy dissipation rate is a fundamental statistic of turbulence, e.g., 
\cite{P00}, \cite{V15}. In 1992, Constantin and Doering \cite{CD92}
established a direct link between phenomenology and NSE predicted energy
dissipation. This work builds on \cite{B78}, \cite{H72} (and others) and has
developed in many important directions subsequently e.g.\textit{, \cite{DF02}%
, \cite{H72}, \cite{V15}, \cite{W97}}. For some simple turbulence
models, \emph{a priori} analysis has shown that $avg(\mathbf{\varepsilon }%
)=O(U^{3}/L)$, e.g., \cite{D12}, \cite{D16}, \cite{D18}, \cite{DNL13}, \cite%
{L02}, \cite{L07}, \cite{L16}, \cite{LRS10}, \cite{LST10}, \cite{P17}, \cite%
{P19}, \cite{P19b}. Often these models are significantly simpler than ones
used in practice. For example, most of the models presented in Wilcox \cite%
{W98} evolve to high complexity. Many require different parameterizations
of $l$ and $k^{\prime }$ in different subregions (that must be identified 
\emph{a priori} through previous flow data). Since the number of models seems
to be growing faster than their \emph{a priori} analytical foundation, there is a
need for \emph{a posteriori} model analysis identifying (as herein) computable
quantities for model assessment. 

\section{Notation and preliminaries}

Let $\Omega $ be an open, regular domain in $\mathbb{R}^{d}$ $(d=2\text{ or }%
3)$. The $L^{2}(\Omega )$ norm and the inner product are $\Vert \cdot \Vert $
and $(\cdot ,\cdot )$. Likewise, the $L^{p}(\Omega )$ norms is $\Vert \cdot
\Vert _{L^{p}}$. $C$ represents a generic positive constant independent of $%
\nu ,\mathcal{R}e$, other model parameters, and the flow scales $U,L$\
defined below.

\begin{definition}
The finite and long time averages of a function $\phi (t)$\ are defined by%
\begin{equation*}
\left\langle \phi \right\rangle _{T}=\frac{1}{T}\int_{0}^{T}\phi (t)dt\text{
and }\left\langle \phi \right\rangle _{\infty }=\lim \sup_{T\rightarrow
\infty }\left\langle \phi \right\rangle _{T}.
\end{equation*}
\end{definition}

These satisfy 
\begin{gather}
\left\langle \phi \psi \right\rangle _{T}\leq \left\langle |\phi
|^{2}\right\rangle _{T}^{1/2}\left\langle |\psi |^{2}\right\rangle
_{T}^{1/2},\text{ }\left\langle \phi \psi \right\rangle _{\infty }\leq
\left\langle |\phi |^{2}\right\rangle _{\infty }^{1/2}\left\langle |\psi
|^{2}\right\rangle _{\infty }^{1/2}\text{\ }  \label{eq:CSinTime} \\
\text{and }\left\langle \left\langle \phi \right\rangle _{\infty
}\right\rangle _{\infty }=\left\langle \phi \right\rangle _{\infty }.  \notag
\end{gather}

\begin{definition}
The viscous and turbulent viscosity energy dissipation rate (per unit
volume) are%
\begin{equation*}
\varepsilon _{0}(u)=\frac{1}{|\Omega |}\int_{\Omega }2\nu |\nabla
^{s}u(x,t)|^{2}dx\text{ and }\varepsilon _{turb}(u)=\frac{1}{|\Omega |}%
\int_{\Omega }\nu _{turb}(x,t)|\nabla ^{s}u(x,t)|^{2}dx.
\end{equation*}%
The force, large scale velocity, and length scales, $F,U,L$, are%
\begin{gather}
F=\frac{1}{|\Omega |}||f||\text{, }U=\left\langle \frac{1}{|\Omega |}%
||u||^{2}\right\rangle _{\infty }^{\frac{1}{2}}\text{ , \ }U^{\prime
}=\left\langle \frac{1}{|\Omega |}||u^{\prime }||^{2}\right\rangle _{\infty
}^{\frac{1}{2}}  \label{eq:ULscales} \\
L=\min \left\{ |\Omega |^{\frac{1}{3}},\frac{F}{||\nabla f(\cdot )||_{\infty
}},\frac{F}{\frac{1}{|\Omega |}||\nabla f||^{2}}\right\} .  \notag
\end{gather}
\end{definition}

$L$ has units of length and satisfies%
\begin{equation}
||\nabla f||_{\infty }\leq \frac{F}{L}\text{ and }\frac{1}{|\Omega |}%
||\nabla f||^{2}\leq \frac{F^{2}}{L^{2}}.  \label{eq:FandLproperties}
\end{equation}%
Dimensional consistency (the Kolmogorov-Prandtl relation) requires $\mathbf{%
\nu }_{turb}\mathbf{(}l,k^{\prime }\mathbf{)}$ $=\sqrt{2}\mu l\sqrt{%
k^{\prime }}.$\ Thus, picking $\mathbf{\nu }_{turb}$ means a choice for $%
l(x,t)$ and a model $k_{\text{model}}^{\prime }$ for $k^{\prime }$\ are
induced. Since $k^{\prime }=\frac{1}{2}|u^{\prime }|^{2}$ this determines a
model for $|u^{\prime }|\simeq |u^{\prime }|_{\text{model}}=\sqrt{2k_{\text{%
model}}^{\prime }}$.

\begin{definition}
Define the velocity scales $U,U^{\prime },U_{\text{model}}^{\prime }$ by%
\begin{equation*}
U=\left\langle \frac{1}{|\Omega |}||u||^{2}\right\rangle _{\infty
}^{1/2}{},U_{\text{model}}^{\prime }=\left\langle \frac{1}{|\Omega |}%
\int_{\Omega }2k^{\prime }dx\right\rangle _{\infty }^{1/2}{}\text{ and }%
U^{\prime }=\left\langle \frac{1}{|\Omega |}\int_{\Omega }|u^{\prime
}|^{2}dx\right\rangle _{\infty }^{1/2}{}.
\end{equation*}
\end{definition}

It has not been necessary herein to specify the initial average leading to
the eddy viscosity term and used to define $U^{\prime }$. Our intuition is
that for a properly defined (and commonly used) averaging operations $%
U^{\prime }\leq U$ and thus $0\leq I(u )\leq 1$.

\begin{definition}
The models' predicted \textbf{turbulent intensity} is 
\begin{equation*}
\text{ }I_{\text{model}}(u)=\left( \frac{U_{\text{model}}^{\prime }}{U}%
\right) ^{2}.
\end{equation*}%
The average \textbf{model length-scale} and \textbf{average turbulent
viscosity} are%
\begin{eqnarray*}
avg(l) &=&\left\langle \frac{1}{|\Omega |}||l(x,t)||^{2}\right\rangle
_{\infty }^{1/2}{}, \\
avg(\mathbf{\nu }_{T}) &=&\left\langle \frac{1}{|\Omega |}\int_{\Omega }|%
\mathbf{\nu }_{turb}(x,t)|dx\right\rangle _{\infty }.
\end{eqnarray*}
\end{definition}

\section{\textbf{Is an eddy viscosity model necessary?}}

This is a question that can only be sensibly asked after discretization in
space and with $\mathbf{\nu }_{T}=0$. (Thus in this section $U$ represents
the NSE velocity scale.) For the chosen numerical (spacial) discretization,
we assume that (i) \textit{no model or numerical dissipation
is present} (A1 below), (ii)\textit{\ the largest discrete gradient
representable is proportional to }$1/meshwidth$ (A2 below, see \cite{C02}, \cite{HH92}, 
 and \cite{Z68} for proofs in specific settings) and, as kinetic
energy is concentrated in the largest scales, (iii) \textit{the discrete
kinetic energy is comparable to the true kinetic energy} (A3 below).

\textbf{A1. [No model or numerical dissipation]} \textit{The total energy
dissipation rate of} $u^{h}$\ \textit{is} $\varepsilon _{0}(u^{h})$.

\textbf{A2. [Inverse Assumption}] \textit{There is a parameter }$h=\triangle
x$\textit{, representing a typical meshwidth, and an }$O(1)$\textit{\
constant }$C_{I}$\textit{\ such that for all discrete velocities }$u^{h}$%
\begin{equation*}
||\nabla ^{s}u^{h}||\leq C_{I}h^{-1}||u^{h}||.
\end{equation*}

\textbf{A3. [Assumption on energy of approximate velocity]}. \textit{There
are constants }$c_{E}$\textit{, }$C_{E}$\textit{\ such that the kinetic
energy of the true and approximate velocities satisfy}%
\begin{equation*}
0<c_{E}\leq \frac{U_{h}}{U}=\sqrt{\frac{\left\langle
||u^{h}||^{2}\right\rangle _{\infty }}{\left\langle ||u||^{2}\right\rangle
_{\infty }}}\leq C_{E}<\infty .
\end{equation*}

\bigskip

\textbf{Definition 2.1.} \textit{The Taylor microscale }$\lambda _{T}$ (%
\textit{e.g., \cite{A98}, \cite{D15}, \cite{P00}, \cite{T35})} \textit{of
the fluid velocity }$u(x,t)$\textit{\ is}%
\begin{equation}
\lambda _{T}(u):=\left( \frac{\frac{1}{15}\left\langle ||\nabla
u||^{2}\right\rangle _{\infty }}{\left\langle ||u||^{2}\right\rangle
_{\infty }}\right) ^{-1/2}.  \label{eq:TaylorDefinition}
\end{equation}

For fully developed, $3d$ turbulent flows (away from walls), it is known, {e.g., \cite{A98}, \cite{D15}, \cite{P00}}, that $\lambda _{T}$ is
significantly larger than the Kolmogorov microscale and scales with the
Reynolds number as%
\begin{equation}  \label{taylor_microscale_scaling}
\lambda _{T}\simeq \mathcal{R}e^{-1/2}L.
\end{equation}%
The Taylor microscale $\lambda _{T}(u)$\ represents an average length of the
velocity $u$. For example, one can have $\mathcal{R}e>>1$\ , but $\lambda
_{T}=O(1)$\ for artificially constructed/manufactured laminar velocities,
such as the Taylor-Green vortex \cite{B05}, \cite{GT37}.

We then have the following.

\begin{theorem}
\textit{Let A1, A2 and A3 hold. If the meshwidth }$h>>2(C_{I}C_{E})\sqrt{15}%
\mathcal{R}e^{-1/2}L,$\textit{\ then }%
\begin{equation*}
\left\langle \varepsilon (u^{h})\right\rangle _{\infty }<<\frac{U^{3}}{L}%
\text{ and }\left\langle \varepsilon (u^{h})\right\rangle _{\infty
}\rightarrow 0\text{ as }\mathcal{R}e\rightarrow \infty .
\end{equation*}

\textit{Contrarily, }$\left\langle \varepsilon (u^{h})\right\rangle _{\infty
}\simeq \frac{U^{3}}{L}$\textit{\ if the Taylor microscale of the computed
solution }$u^{h}$\textit{\ satisfies }%
\begin{equation*}
\lambda _{T}(u^{h})\leq \frac{\sqrt{30}}{2}\mathcal{R}e^{-1/2}L.
\end{equation*}
\end{theorem}

\begin{proof}
\textit{By A1, A2}%
\begin{eqnarray*}
\left\langle \varepsilon (u^{h})\right\rangle _{\infty } &=&2\nu
\left\langle ||\nabla ^{s}u^{h}||^{2}\right\rangle _{\infty }\leq 2\nu
C_{I}^{2}h^{-2}\left\langle ||u^{h}||^{2}\right\rangle _{\infty } \\
&\leq &2\nu C_{I}^{2}h^{-2}U_{h}^{2}=\frac{\nu }{LU}C_{I}^{2}\left( \frac{h}{%
L}\right) ^{-2}\left( \frac{U_{h}}{U}\right) ^{2}\frac{U^{3}}{L} \\
&\leq &2\left[ \mathcal{R}e^{-1}C_{I}^{2}C_{E}^{2}\left( \frac{h}{L}\right)
^{-2}\right] \frac{U^{3}}{L}\text{, by A3.}
\end{eqnarray*}%
\textit{Thus, the first case of under-dissipation occurs when the bracketed
term}%
\begin{equation*}
\mathcal{R}e^{-1}C_{I}^{2}C_{E}^{2}\left( \frac{h}{L}\right) ^{-2}<<\frac{1}{%
2}\text{ }\Leftrightarrow \text{ }h>>\sqrt{2}\left( C_{I}C_{E}\right) 
\mathcal{R}e^{-1/2}L=\mathcal{O}\left( \lambda _{T}(u)\right) .
\end{equation*}%
\textit{For the second claim, by A1, A3,}%
\begin{gather*}
\left\langle \varepsilon (u^{h})\right\rangle _{\infty }=2\nu \left\langle
||\nabla ^{s}u^{h}||^{2}\right\rangle _{\infty }=2\nu \frac{\left\langle
||\nabla ^{s}u^{h}||^{2}\right\rangle _{\infty }}{\left\langle
||u^{h}||^{2}\right\rangle _{\infty }}\left\langle
||u^{h}||^{2}\right\rangle _{\infty } \\
=30\frac{\nu }{LU}\lambda _{T}(u^{h})^{-2}LUU_{h}^{2} \\
=30\mathcal{R}e^{-1}\left( \frac{\lambda _{T}(u^{h})}{L}\right) ^{-2}\left( 
\frac{U_{h}}{U}\right) ^{2}\frac{U^{3}}{L}\leq 30C_{E}^{2}\left[ \mathcal{R}%
e^{-1}\left( \frac{\lambda _{T}(u^{h})}{L}\right) ^{-2}\right] \frac{U^{3}}{L%
}.
\end{gather*}%
\textit{The bracketed term is }$\mathcal{O}(1)$\textit{\ provided }$\lambda
_{T}(u^{h})\simeq \sqrt{30}Re^{-1/2}L$\textit{, as claimed.}
\end{proof}

\section{Does the eddy viscosity model fail?}

\label{sec:EV_failure} The most common failure mode of eddy viscosity models
is model over dissipation. Model dissipation can be studied at the level of
the continuum model \eqref{eq:EVmodel}, that is, without a spacial
discretization. Since this simplifies notation, we do so in this section.
Consider therefore the model \eqref{eq:EVmodel} and recall that the data $%
u_{0}(x),f(x)$\ is smooth, divergence free, and both vanish on $\partial
\Omega $. The next theorem establishes that model dissipation is independent
of solution gradients and controlled by the average of the chosen eddy
viscosity parameterization $avg(\mathbf{\nu }_{turb})$ 
\begin{equation*}
avg(\mathbf{\nu }_{turb})=\left\langle \frac{1}{|\Omega |}\int_{\Omega }|%
\mathbf{\nu }_{turb}(x,t)|dx\right\rangle _{\infty }.
\end{equation*}

\begin{theorem}
\label{thm:turb_intens1} The time averaged rate of total energy dissipation
for the eddy viscosity model satisfies the following. For any $0<\beta <1,$%
\begin{equation*}
\left\langle \varepsilon _{0}+\varepsilon _{turb}\right\rangle \leq \left( 
\frac{2}{2-\beta }+\frac{2}{\beta (2-\beta )}\mathcal{R}e^{-1}+\frac{1}{%
\beta (2-\beta )}\frac{avg(\mathbf{\nu }_{turb})}{LU}\right) \frac{U^{3}}{L}.
\end{equation*}
\end{theorem}

The key term is $\frac{avg(\mathbf{\nu }_{turb})}{LU}$. For this term we can
further separate the effects of the choice of $l$ and $k^{\prime }$ in the
model as follows.

\begin{theorem}
\label{thm:turb_intens2} We have%
\begin{equation}  \label{eqn:sec4-1}
\frac{avg(\mathbf{\nu }_{turb})}{LU}\leq \mu \frac{avg(l)}{L}\sqrt{I_{\text{%
model}}(u)}=\mu \frac{avg(l)}{L}\frac{U_{\text{model}}^{\prime }}{U^{\prime }%
}\sqrt{I(u)}.
\end{equation}
\end{theorem}

As a consequence there follows.

\begin{corollary}
The time averaged energy rate of total energy dissipation for the general
eddy viscosity model satisfies the following. For any $0<\beta <1,$%
\begin{eqnarray*}
\left\langle \varepsilon _{0}+\varepsilon _{turb}\right\rangle _{\infty }
&\leq &\left( \frac{2}{2-\beta }+\frac{2}{\beta (2-\beta )}\mathcal{R}e^{-1}+%
\frac{1}{\beta (2-\beta )}\mu \frac{avg(l_{m})}{L}\frac{U_{m}^{\prime }}{U}%
\right) \frac{U^{3}}{L} \\
&&\text{and } \\
\left\langle \varepsilon _{0}+\varepsilon _{turb}\right\rangle _{\infty }
&\leq &\left( \frac{2}{2-\beta }+\frac{2}{\beta (2-\beta )}\mathcal{R}e^{-1}+%
\frac{1}{\beta (2-\beta )}\mu \frac{avg(l)}{L}\frac{U_{m}^{\prime }}{%
U^{\prime }}\sqrt{I(u)}\right) \frac{U^{3}}{L}.
\end{eqnarray*}
\end{corollary}

\begin{proof}
The claim follows by rearranging the last term in the estimate using the
definition of the turbulent intensity $I(u)=(U^{\prime }/U)^{2}$.
\end{proof}

As noted above, the importance of this result is that the three quantities 
\begin{equation*}
\frac{avg(\mathbf{\nu }_{turb})}{LU},\frac{avg(l)}{L},\frac{U_{m}^{\prime }}{%
U},
\end{equation*}%
are computable. If too large, their spacial distribution can be checked and
the resulting information used to improve the model.

\subsection{Proof of Theorem 1}

From (\ref{eq:EnergyEquality})%
\begin{equation}
\frac{1}{2T}\frac{1}{|\Omega |}||u(T)||^{2}+\left\langle \varepsilon
_{0}+\varepsilon _{turb}\right\rangle _{T}\leq \frac{1}{2T}\frac{1}{|\Omega |%
}||u_{0}||^{2}+\left\langle \frac{1}{|\Omega |}(f,u(t))\right\rangle _{T}\,,
\label{eq:AveregedEnergyIneq}
\end{equation}%
and standard arguments, it follows that, uniformly in $T$,%
\begin{equation}
\sup_{T\in (0,\infty )}||u(T)||^{2}\leq C(data)<\infty \text{ \ and }%
\left\langle \varepsilon _{0}+\varepsilon _{turb}\right\rangle _{T}\leq
C(data)<\infty .  \label{eq:Boundsonuand gradu}
\end{equation}%
For the RHS of the energy inequality, from (\ref{eq:CSinTime}) there follows 
\begin{equation*}
\left\langle \frac{1}{|\Omega |}(f,u(t))\right\rangle _{T}\,\leq F\sqrt{%
\left\langle \frac{1}{|\Omega |}||u(t)||^{2}\right\rangle _{T}},
\end{equation*}%
which, from (\ref{eq:AveregedEnergyIneq}), implies 
\begin{equation}
\left\langle \varepsilon _{0}+\varepsilon _{turb}\right\rangle _{T}\leq 
\mathcal{O}(\frac{1}{T})+F\left\langle \frac{1}{|\Omega |}%
||u||^{2}\right\rangle _{T}^{\frac{1}{2}}.  \label{eq:NewStep1}
\end{equation}%
To bound $F$ in terms of flow quantities, take the inner product of the
model (\ref{eq:EVmodel}) with $f(x)$, integrate by parts (using $\nabla
\cdot f=0$ and $f(x)=0$ on $\partial \Omega $), and average over $[0,T]$. This
gives%
\begin{gather}
F^{2}=\frac{(u(T)-u_{0},f)}{T|\Omega |}-\left\langle \frac{1}{|\Omega |}%
(uu,\nabla f)\right\rangle _{T}  \label{eq:Step2} \\
+\ \left\langle \frac{1}{|\Omega |}\int_{\Omega }2\nu \nabla ^{s}u:\nabla
^{s}f+\mathbf{\nu }_{turb}(x,t)\nabla ^{s}u:\nabla ^{s}fdx\right\rangle _{T}.
\notag
\end{gather}%
Analysis of the first three terms on the RHS parallels the NSE case in, e.g.%
\textit{, \cite{CD92}, \cite{DF02}, \cite{H72}, \cite{V15}, \cite{W97}}.
The fourth is the key, model-specific term. The first term on the RHS is\ $%
\mathcal{O}(1/T)$ by (\ref{eq:Boundsonuand gradu}). The second is bounded by
Holders inequality, (\ref{eq:CSinTime}), and (\ref{eq:FandLproperties}) as
follows%
\begin{gather*}
\left\langle \frac{1}{|\Omega |}(uu,\nabla f)\right\rangle _{T}\leq
\left\langle ||\nabla f(\cdot )||_{\infty }\frac{1}{|\Omega |}||u(\cdot
,t)||^{2}\right\rangle _{T} \\
\leq ||\nabla f(\cdot )||_{\infty }\left\langle \frac{1}{|\Omega |}||u(\cdot
,t)||^{2}\right\rangle _{T}\leq \frac{F}{L}\left\langle \frac{1}{|\Omega |}%
||u(\cdot ,t)||^{2}\right\rangle _{T}.
\end{gather*}%
The third term is bounded by analogous steps to the second. For any $0<\beta
<1$, we have%
\begin{gather*}
\left\langle \frac{1}{|\Omega |}\int_{\Omega }2\nu \nabla ^{s}u(x,t):\nabla
^{s}f(x)dx\right\rangle _{T} \\
\leq \left\langle \frac{4\nu ^{2}}{|\Omega |}||\nabla
^{s}u||^{2}\right\rangle _{T}^{\frac{1}{2}}\left\langle \frac{1}{|\Omega |}%
||\nabla ^{s}f||^{2}\right\rangle _{T}^{\frac{1}{2}}\leq \left\langle
\varepsilon _{0}\right\rangle _{T}^{\frac{1}{2}}\sqrt{2\nu }\frac{F}{L}\leq 
\frac{\beta F}{2U}\left\langle \varepsilon _{0}\right\rangle _{T}+\frac{UF}{%
\beta }\frac{\nu }{L^{2}}.
\end{gather*}%
The fourth, \textit{model dependent} term, is estimated successively as
follows%
\begin{gather*}
\left\langle \frac{1}{|\Omega |}\int_{\Omega }\mathbf{\nu }_{turb}\nabla
^{s}u(x,t):\nabla ^{s}f(x)dx\right\rangle _{T}\leq \left\langle \frac{1}{%
|\Omega |}\int_{\Omega }\sqrt{\mathbf{\nu }_{turb}}\left( \sqrt{\mathbf{\nu }%
_{turb}}|\nabla ^{s}u|\right) |\nabla ^{s}f|dx\right\rangle _{T} \\
\leq ||\nabla ^{s}f||_{L\infty }\left\langle \left( \frac{1}{|\Omega |}%
\int_{\Omega }\mathbf{\nu }_{turb}dx\right) ^{1/2}\left( \frac{1}{|\Omega |}%
\int_{\Omega }\mathbf{\nu }_{turb}|\nabla ^{s}u|^{2}dx\right)
^{1/2}dx\right\rangle _{T} \\
\leq ||\nabla ^{s}f||_{L\infty }\left\langle \left( \frac{1}{|\Omega |}%
\int_{\Omega }\mathbf{\nu }_{turb}dx\right) ^{1/2}\varepsilon
_{turb}^{1/2}\right\rangle _{T} \\
\leq \frac{F}{L}\left( \frac{U}{F}\left\langle \frac{1}{|\Omega |}%
\int_{\Omega }\mathbf{\nu }_{turb}dx\right\rangle _{T}\right) ^{1/2}\left( 
\frac{F}{U}\left\langle \varepsilon _{turb}\right\rangle _{T}\right) ^{1/2}
\\
\leq \frac{\beta }{2}\frac{F}{U}\left\langle \varepsilon
_{turb}\right\rangle _{T}+\frac{1}{2\beta }\frac{UF}{L^{2}}\left\langle 
\frac{1}{|\Omega |}\int_{\Omega }\mathbf{\nu }_{turb}dx\right\rangle _{T}.
\end{gather*}%
Using these estimates in the bound for $F^{2}$ yields%
\begin{eqnarray*}
F^{2} &\leq &\mathcal{O}\left( \frac{1}{T}\right) +\frac{F}{L}\left\langle 
\frac{1}{|\Omega |}||u||^{2}\right\rangle _{T}+\frac{\beta }{2}%
U^{-1}F\left\langle \varepsilon _{0}\right\rangle _{T}+\frac{1}{\beta }UF%
\frac{\nu }{L^{2}} \\
&&+\frac{\beta }{2}\frac{F}{U}\left\langle \varepsilon _{turb}\right\rangle
_{T}+\frac{1}{2\beta }\frac{UF}{L^{2}}\left\langle \frac{1}{|\Omega |}%
\int_{\Omega }\mathbf{\nu }_{turb}dx\right\rangle _{T}.
\end{eqnarray*}%
Thus, we have an estimate for $F\left\langle \frac{1}{|\Omega |}%
||u||^{2}\right\rangle _{T}^{\frac{1}{2}}$%
\begin{gather*}
F\left\langle \frac{1}{|\Omega |}||u||^{2}\right\rangle _{T}^{\frac{1}{2}%
}\leq \mathcal{O}\left( \frac{1}{T}\right) +\frac{1}{L}\left\langle \frac{1}{%
|\Omega |}||u||^{2}\right\rangle _{T}^{\frac{3}{2}}+\frac{\beta }{2}\frac{%
\left\langle \frac{1}{|\Omega |}||u||^{2}\right\rangle _{T}^{\frac{1}{2}}}{U}%
\left\langle \varepsilon _{0}\right\rangle _{T} \\
+\frac{1}{\beta }\left\langle \frac{1}{|\Omega |}||u||^{2}\right\rangle
_{T}^{\frac{1}{2}}U\frac{\nu }{L^{2}}+\frac{\beta }{2}\frac{\left\langle 
\frac{1}{|\Omega |}||u||^{2}\right\rangle _{T}^{\frac{1}{2}}}{U}\left\langle
\varepsilon _{turb}\right\rangle _{T} \\
+\frac{1}{2\beta }\left\langle \frac{1}{|\Omega |}||u||^{2}\right\rangle
_{T}^{\frac{1}{2}}\frac{U}{L^{2}}\left\langle \frac{1}{|\Omega |}%
\int_{\Omega }\mathbf{\nu }_{turb}dx\right\rangle _{T}.
\end{gather*}

These four estimates then imply that%
\begin{gather*}
\left[ 1-\frac{\beta }{2}\frac{\left\langle \frac{1}{|\Omega |}%
||u||^{2}\right\rangle _{T}^{\frac{1}{2}}}{U}\right] \left\langle
\varepsilon _{0}+\varepsilon _{turb}\right\rangle _{T} \\
\leq \mathcal{O}\left( \frac{1}{T}\right) +\frac{1}{L}\left\langle \frac{1}{%
|\Omega |}||u||^{2}\right\rangle _{T}^{\frac{3}{2}}+\frac{1}{\beta }%
\left\langle \frac{1}{|\Omega |}||u||^{2}\right\rangle _{T}^{\frac{1}{2}}U%
\frac{\nu }{L^{2}}+ \\
+\frac{1}{2\beta }\left\langle \frac{1}{|\Omega |}||u||^{2}\right\rangle
_{T}^{\frac{1}{2}}\frac{U}{L^{2}}\left\langle \frac{1}{|\Omega |}%
\int_{\Omega }\mathbf{\nu }_{turb}dx\right\rangle _{T}.
\end{gather*}%
The limit superior as $T\rightarrow \infty $, which exists by\ (\ref%
{eq:Boundsonuand gradu}), yields the following%
\begin{gather*}
\left[ 1-\frac{\beta }{2}\right] \left\langle \varepsilon _{0}+\varepsilon
_{turb}\right\rangle _{\infty }\leq \frac{U^{3}}{L}+\frac{1}{\beta }U^{2}%
\frac{\nu }{L^{2}}+\frac{avg(\mathbf{\nu }_{turb})}{2\beta }\frac{U^{2}}{%
L^{2}} \\
\leq \frac{U^{3}}{L}\left( 1+\frac{1}{\beta }\frac{\nu }{LU}+\frac{1}{2\beta 
}\frac{avg(\mathbf{\nu }_{turb})}{LU}\right) .
\end{gather*}%
Thus, after rearranging,%
\begin{equation*}
\left\langle \varepsilon _{0}+\varepsilon _{turb}\right\rangle _{\infty
}\leq \frac{U^{3}}{L}\left( \frac{2}{2-\beta }+\frac{2}{\beta (2-\beta )}%
\mathcal{R}e^{-1}+\frac{1}{\beta (2-\beta )}\frac{avg(\mathbf{\nu }_{turb})}{%
LU}\right).
\end{equation*}

\subsection{Proof of Theorem 2: estimating $\frac{avg(\mathbf{\protect\nu }%
_{turb})}{LU}$}

We now prove the estimate in Theorem 2 for $avg(\mathbf{\nu }_{turb})$.
Since $\mathbf{\nu }_{turb}\mathbf{=}\sqrt{2}\mu l\sqrt{\frac{1}{2}%
|u^{\prime }|_{\text{model}}^{2}}$ we have 
\begin{eqnarray*}
\frac{1}{LU}\left\langle \frac{1}{|\Omega |}\int_{\Omega }\mathbf{\nu }%
_{turb}(x,t)dx\right\rangle _{T} &=&\frac{1}{LU}\left\langle \frac{1}{%
|\Omega |}\int_{\Omega }\sqrt{2}\mu l\sqrt{\frac{1}{2}|u^{\prime }|_{\text{%
model}}^{2}}dx\right\rangle _{T} \\
&=&\frac{\mu }{LU}\left\langle \frac{1}{|\Omega |}\int_{\Omega }l|u^{\prime
}|_{\text{model}}dx\right\rangle _{T}.
\end{eqnarray*}%
By the Cauchy-Schwarz inequality in space and (\ref{eq:CSinTime}) we have%
\begin{equation}
\frac{1}{LU}\left\langle \frac{1}{|\Omega |}\int_{\Omega }\mathbf{\nu }%
_{turb}dx\right\rangle _{T}\leq \frac{\mu }{LU}\left\langle \frac{1}{|\Omega
|}||l||^{2}\right\rangle _{T}^{1/2}\left\langle \frac{1}{|\Omega |}%
|||u^{\prime }|_{\text{model}}||^{2}\right\rangle _{T}{}^{1/2}.
\label{eq:CSestimateForNusubT}
\end{equation}%
Taking the limit superior of (\ref{eq:CSestimateForNusubT}) gives, as
claimed, 
\begin{eqnarray*}
\frac{avg(\mathbf{\nu }_{turb})}{LU} &\leq &\frac{\mu }{LU}avg(l)U_{\text{%
model}}^{\prime }=\mu \frac{avg(l)}{L}\frac{U_{\text{model}}^{\prime }}{U} \\
&=&\mu \frac{avg(l)}{L}\sqrt{I_{\text{model}}(u)}=\mu \frac{avg(l)}{L}\frac{%
U_{\text{model}}^{\prime }}{U^{\prime }}\sqrt{I(u)}.
\end{eqnarray*}

\section{Conclusions and open problems}

One basic challenge is that the analysis of models has advanced more slowly
than new models have been developed to respond to the needs of predictive
flow simulations. This means that models can evolve by more complex
parameterizations rather than more careful representation of the effects of
fluctuations on mean velocities. The gap between model complexity and model
understanding is widening even further due to the current model development
using machine learning and neural networks based eddy viscosity models. Since turbulence
models are used in many safety critical settings, there is an obvious need
to assess models during a simulation. To this end, this report presents an
alternative approach to assess model dissipation. The first result is that,
surprisingly, the need for eddy viscosity depends on the mesh resolving the
Taylor microscale rather than the Kolmogorov micro-scale. The second result
is when an eddy viscosity model is used, its total dissipation can be estimated in
terms of several computable flow statistics. When the model over dissipates,
these can be used to isolate the part of the model needing improvement; the
estimates separate the effects of the different model choices so that, when
over-dissipation occurs, the source in the various modeling decisions can be
isolated.

Open problems abound. Our analysis assumes that $f(x)=0$ on the
boundary. This means the effect of boundary layers is less than small scales
generated by the system nonlinearity. To seek the right computable
statistics for turbulent boundary layers, an analysis of energy dissipation
for shear flows is needed. There is a small number of eddy viscosity models
where quantities like the turbulent statistics identified herein can be
performed. Expanding this list to models closer to those used in practice is
an important collection of open problems. Numerical dissipation often is
much greater than model dissipation. Thus, analysis including numerical
dissipation is of great importance. Estimation of the effect of eddy
viscosity terms on helicity dissipation rates is little studied but possibly
critical for correct predictions of rotational flows. Neural network based
eddy viscosity models are at a beginning point in their development. Thus, practically
any question (analytical, theoretical of experimental) known for classic
models is open for these.


\begin{thebibliography}{COPPM11}
\bibitem[A98]{A98} \textsc{V.S. Arpaci}, \emph{Microscales of turbulence,}
CRC Press, 1998.

\bibitem[B05]{B05} \textsc{L.C. Berselli}, \emph{On the large eddy
simulation of the Taylor-Green vortex}, Journal of Mathematical Fluid Mechanics, 7 (2005),
164-191.



\bibitem[B78]{B78} \textsc{F.H. Busse}, \emph{The optimum theory of
turbulence}, Advances in Applied Mechanics, 18 (1978), 77-121.


\bibitem[COPPM11]{COPPM11} \textsc{S.H. Cheung, T.A. Oliver, E.E. Prudencio,
S. Prudhomme, and R.D. Moser}, \emph{Bayesian uncertainty analysis with
applications to turbulence modeling}, Reliability Engineering and System
Safety, 96 (2011), 1137-1149.

\bibitem[C02]{C02} \textsc{P.G. Ciarlet}, \emph{The finite element method
for elliptic problems}, SIAM, 2002.

\bibitem[CD92]{CD92} \textsc{P. Constantin and C. Doering}, \emph{Energy
dissipation in shear driven turbulence}, Phys. Rev. Letters, 69 (1992),
1648-1651.

\bibitem[D12]{D12} \textsc{A.A. Dunca,} \emph{A two-level multiscale deconvolution method for
the large eddy simulation of turbulent flows,}  Mathematical Models and
Methods in Applied Sciences. 22 (2012), 1250001.

\bibitem[D15]{D15} \textsc{P. Davidson}, \emph{Turbulence: an introduction
	for scientists and engineers,} Oxford Univ. Press, 2015.

\bibitem[D16]{D16} \textsc{A.A. Dunca} \emph{On an energy inequality for the
approximate deconvolution models,} Nonlinear Analysis: Real World
Applications, 32 (2016), 294-300.

\bibitem[D18]{D18} \textsc{A.A. Dunca} \emph{Estimates of the modeling error of the $%
\alpha $-models of turbulence in two and three space dimensions,} Journal of
Mathematical Fluid Mechanics, 20 (2018), 1123-1135.


\bibitem[DF02]{DF02} \textsc{C. Doering and C. Foias}, \emph{Energy
dissipation in body-forced turbulence,} Journal of
Mathematical Fluid Mechanics, 467 (2002), 289-306.


\bibitem[DG91]{DG91} \textsc{Q. Du and M. Gunzburger}, \emph{Analysis of a
	Ladyzhenskaya model for incompressible viscous flow,} JMAA, 155 (1991), 21-45.

\bibitem[DG95]{DG95} \textsc{C.R. Doering and J.D. Gibbon}, \emph{Applied
analysis of the Navier-Stokes equations,} Cambridge Univ. Press,
1995.

\bibitem[DNL13]{DNL13} \textsc{A. A. Dunca, M .Neda, and L.G. Rebholz.} \emph{A
mathematical and numerical study of a filtering-based multiscale fluid model
with nonlinear eddy viscosity,} Computers \& Mathematics with Applications,
66 (2013), 917-933.

\bibitem[GT37]{GT37} \textsc{A. E. Green and G. I. Taylor}, \emph{Mechanism of
the production of small eddies from larger ones}, Proc. Royal Soc. A., 158
(1937), 499-521.

\bibitem[H72]{H72} \textsc{L . N. Howard}, \emph{Bounds on flow quantities},
Ann. Rev. Fluid Mech., 4 (1972), 473-494.

\bibitem[HH92]{HH92} \textsc{I. Harari and T.J.R. Hughes}. \emph{What are }$%
C $\emph{\ and }$h$\emph{ inequalities for the analysis and design of
finite element methods}, Computers \& Mathematics with Applications, 97 (1992), 157-192.



\bibitem[L02]{L02} \textsc{W. Layton}, \emph{Bounds on energy dissipation
rates of large eddies in turbulent shear flows}, Math. and Comp. Modeling,
35 (2002), 1445-1451.

\bibitem[L07]{L07} \textsc{W. Layton}, \emph{Bounds on energy and helicity dissipation rates
of approximate deconvolution models of turbulence,} SIAM Journal on
Mathematical Analysis, 39 (2007), 916-31.

\bibitem[L16]{L16} \textsc{W. Layton}, \emph{Energy dissipation in the Smagorinsky model of
turbulence,} Applied Mathematics Letters, 59 (2016), 56-59.

\bibitem[LKT16]{LKT16} \textsc{J. Ling, A. Kurzawski, and J. Templeton}, 
\emph{Reynolds averaged turbulence modelling using deep neural networks with
embedded invariance,} Journal of Fluid Mechanics. 807 (2016), 155-166.


\bibitem[LRS10]{LRS10} \textsc{W. Layton, L.G. Rebholz, and  M. Sussman} \emph{Energy and helicity
dissipation rates of the NS-alpha and NS-alpha-deconvolution models,} IMA
Journal of Applied Mathematics, 75 (2010), 56-74.
%

\bibitem[LST10]{LST10} \textsc{W. Layton, M. Sussman, and C. Trenchea}, \emph{Bounds on energy,
magnetic helicity, and cross helicity dissipation rates of approximate
deconvolution models of turbulence for MHD flows,} Numerical Functional
Analysis and Optimization, 31 (2010), 577-595.



\bibitem[MX06]{MX06} \textsc{W.J. Martin and M. Xue}, \emph{Initial
condition sensitivity analysis of a mesoscale forecast using very-large
ensembles}, Mon. Wea. Rev., 134 (2006), 192-207.

\bibitem[P00]{P00} \textsc{S.\ Pope}, \emph{Turbulent Flows,} Cambridge
Univ. Press, 2000.

\bibitem[P17]{P17} \textsc{A. Pakzad,} \emph{Damping functions correct over-dissipation of
the Smagorinsky model,} Mathematical Methods in the Applied Sciences, 40
(2017), 5933-5945.

\bibitem[P19]{P19} \textsc{A. Pakzad,} \emph{Analysis of mesh effects on turbulent flow
statistics,} Journal of Mathematical Analysis and Applications, 475 (2019), 839-860.

\bibitem[P19b]{P19b} \textsc{A. Pakzad,} \emph{On the long time behavior of time relaxation
model of fluids,} arXiv preprint arXiv:1903.12339, 2019 Mar 29.


\bibitem[RL14]{RL14} \textsc{T.C. Rebollo and R. Lewandowski,} \emph{Mathematical and
numerical foundations of turbulence models and applications,} Springer, 2014.


\bibitem[S68]{S68} \textsc{V. P. Starr}, \emph{Physics of negative viscosity
phenomena,} McGraw-Hill, 1968

\bibitem[S01]{S01} \textsc{P. Sagaut,} \emph{Large eddy simulation for
incompressible flows,} Springer, 2002.

\bibitem[SALL19]{SALL19} \textsc{L. Sun, W. An, X. Liu, and H. Lyu}, \emph{On developing data-driven turbulence model for DG solution of RANS}, Chinese
Journal of Aeronautics, 32 (2019), 869-1884.

\bibitem[T35]{T35} \textsc{G.I. Taylor}, \emph{Statistical theory of
turbulence,} Proc. Royal Soc. London. Series A, 151 (1935),
421-464.

\bibitem[TK93]{TK93} \textsc{Z. Toth and E. Kalney}, \emph{Ensemble
forecasting at NMC: The generation of perturbations,} Bull. Amer. Meteor.
Soc., 74 (1993), 2317-2330.

\bibitem[V15]{V15} \textsc{J.C. Vassilicos}, \emph{Dissipation in turbulent
flows}, Ann. Rev. Fluid Mech., 47 (2015), 95-114.

\bibitem[W97]{W97} \textsc{X. Wang}, \emph{The time averaged energy
dissipation rates for shear flows,} Physica D, 99 (1997), 555-563.

\bibitem[W98]{W98} \textsc{D. C. Wilcox}, \emph{Turbulence modeling for
CFD,} DCW Industries, 1998.


\bibitem[Z68]{Z68} \textsc{M. Zl\'{a}mal}, \emph{On the finite element method,} Numer. Math., 12 (1968), 394-409.
\end{thebibliography}
\end{document}